\documentclass{article}
\usepackage{multicol,lipsum}
\usepackage{microtype}
\usepackage{bbm}

\usepackage{graphicx}
\usepackage{subfigure}
\usepackage{amsmath}
\usepackage{dcolumn}
\usepackage{amssymb}
\usepackage{amsthm,multirow}
\usepackage{booktabs} 
\newtheorem{theorem}{Theorem}
\newtheorem{definition}{Definition}
\newtheorem{lemma}{Lemma}

\newcommand{\dw}{\mathsf{D}^{(1)}_{\mathsf{W}}}
\newcommand{\dwtwo}{\mathsf{D}^{(2)}_{\mathsf{W}}}
\newcommand{\perm}{\mathcal{S}}
\usepackage{dsfont}
\usepackage{algorithm}
\usepackage{algorithmicx}
\usepackage{algpseudocode}
\usepackage{mathtools}
\usepackage{hyperref}

\usepackage{arxiv}

\usepackage[utf8]{inputenc} 
\usepackage[T1]{fontenc}    
\usepackage{hyperref}       
\usepackage{url}            
\usepackage{booktabs}       
\usepackage{amsfonts}       
\usepackage{nicefrac}       
\usepackage{microtype}      
\usepackage{lipsum}
\usepackage[xspace]{ellipsis}
\usepackage{algorithm}
\newtheorem{assumption}{Assumption}
\usepackage[colorinlistoftodos,prependcaption,textsize=tiny,textwidth=20mm]{todonotes}

\newcommand{\order}[1]{O\left(#1\right)}
\newcommand{\otilde}[1]{\widetilde{O}\left(#1\right)}

\newcommand{\sgdwor}{SGDo\xspace}
\newcommand{\sgd}{SGD\xspace}
\newcommand{\norm}[1]{\left\| #1 \right\|}
\newcommand{\defeq}{\stackrel{\mathrm{def}}{=}}
\newcommand{\R}{\mathbb{R}}

\title{SGD without Replacement: Sharper Rates for General Smooth Convex Functions}

\author{
   Prateek Jain \\
  Microsoft Research\\
  Bengaluru, India \\
  \texttt{prajain@microsoft.com} \\
  \And
  Dheeraj Nagaraj\thanks{} \\
  Department of Electrical Engineering and Computer Science\\
  Massachusetts Institute of Technology\\
  Cambridge, USA 02139 \\
  \texttt{dheeraj@mit.edu} \\
   \And
  Praneeth Netrapalli \\
  Microsoft Research\\
  Bengaluru, India \\
  \texttt{praneeth@microsoft.com} \\
}

\begin{document}
\maketitle

\begin{abstract}
We study stochastic gradient descent {\em without replacement} (\sgdwor) for smooth convex functions. \sgdwor is widely observed to converge faster than true \sgd where each sample is drawn independently {\em with replacement}~\cite{bottou2009curiously} and hence, is more popular in practice. But it's convergence properties are not well understood as sampling without replacement leads to coupling between iterates and gradients. By using method of exchangeable pairs to bound Wasserstein distance, we provide the first non-asymptotic results for \sgdwor when applied to {\em general smooth, strongly-convex} functions. In particular, we show that \sgdwor converges at a rate of $O(1/K^2)$ while \sgd~is known to converge at $O(1/K)$ rate, where $K$ denotes the number of passes over data and is required to be {\em large enough}. Existing results for \sgdwor in this setting require additional {\em Hessian Lipschitz assumption}~\cite{gurbuzbalaban2015random,haochen2018random}. 

	For {\em small} $K$, we show \sgdwor can achieve same convergence rate as \sgd for {\em general smooth strongly-convex} functions. Existing results in this setting require $K=1$ and hold only for generalized linear models \cite{shamir2016without}. 	 In addition, by careful analysis of the coupling, for both large and small $K$, we obtain better dependence on problem dependent parameters like condition number. 
\end{abstract}

\keywords{Stochastic Gradient Descent \and Random Reshuffling \and Machine Learning \and Convex Optimization \and Without Replacement Sampling}

\section{Introduction}


In this paper, we study the standard finite sum optimization problem that arises in most machine learning based optimization problems: 
$$F(x) := \frac{1}{n}\sum_{i=1}^n f(x;i),$$
where $f(x;i):\mathbb{R}^d \rightarrow \mathbb{R}$ is the $i$-th component function. For example, in standard ERM and deep learning training $f(x; i)$ denotes the loss function w.r.t. the $i$-th data point. Stochastic Gradient Descent (SGD), originally proposed by \cite{robbins1951textordfemininea}, has emerged as one of the most popular techniques to solve this problem. 

At $t$-th step, SGD updates the iterate by $x_{t+1}=x_t-\eta \nabla f(x_t;i_t)$ where $\nabla f(x_t;i_t)$ is the gradient of $f(x_t;i_t)$ and $i_t$ is selected uniformly at random {\em with replacement} yielding $\mathbb{E}_{i_t}[\nabla f(x_t; i_t)]=\nabla F(x_t)$.  SGD has been extensively studied in literature and a vast number of results are known in many different settings, most prominent being that of convex optimization~\cite{benaim1999dynamics,borkar2009stochastic,kushner2003stochastic,bubeck2015convex,bottou2018optimization}.
\begin{table*}[ht!]
	\vskip 0.15in
	\begin{center}
		\begin{small}
			\begin{sc}
				\begin{tabular}{|l|c|c|c|r|}
					\toprule
					Paper & Guarantee & Assumptions & Step Sizes  \\
					\midrule
					\cite{gurbuzbalaban2015random}& $O\left(\tfrac{C(n,d)}{K^2}\right)$& \multirow{ 2}{*}{\begin{tabular}{@{}c@{}}\scriptsize Lipschitz, Strong convexity \\ \scriptsize Smoothness, \textbf{Hessian Lipschitz} \\ \scriptsize $K > \kappa^{1.5} \sqrt{n}$ \end{tabular}}& $\frac{1}{K}$ \\
					\cline{1-2} \cline{4-4}
					\cite{haochen2018random} & $\tilde{O}\left(\tfrac{1}{n^2K^2} + \tfrac{1}{K^3} \right)$ & &$\frac{\log{nK}}{\mu nK}$\\
					\midrule
					\textbf{This paper, Theorem~\ref{thm:main_result_2}} & $\tilde{O}\left(\tfrac{1}{nK^2}\right)$ &\shortstack{\begin{tabular}{@{}c@{}}\scriptsize Lipschitz, Strong convexity \\ \scriptsize Smoothness, $K > \kappa^{2}\log{nK}$ \end{tabular}} & $\frac{\log{nK}}{\mu nK}$\\
										\midrule
										\midrule
					\cite{shamir2016without} & $O\left(\tfrac{1}{nK}\right)$ & \shortstack{{\scriptsize Lipschitz, Strong convexity}\\ {\scriptsize Smoothness} \\ {\scriptsize \textbf{Generalized Linear Function}, $K=1$ }} & $\frac{1}{\mu nK}$  \\
					\midrule
					\textbf{This Paper, Theorem~\ref{thm:main_result_3}} & $O\left(\tfrac{1}{nK} \right)$ & \shortstack{{\scriptsize Lipschitz, Strong convexity}\\ {\scriptsize Smoothness} } & $\min\left(\tfrac{2}{L}, \tfrac{\log{nK}}{\mu nK}\right) $\\
					\midrule
					\midrule
					\cite{shamir2016without} & $O\left(\tfrac{1}{\sqrt{nK}}\right)$ & \shortstack{{\scriptsize Lipschitz}\\ {\scriptsize \textbf{Generalized Linear Function}, $K=1$ }} & $\frac{1}{\sqrt{nK}}$  \\
					\midrule
					\textbf{This Paper, Theorem~\ref{thm:main_result_1}} & $O\left(\tfrac{1}{\sqrt{nK}} \right)$ & \shortstack{{\scriptsize Lipschitz, \textbf{Smoothness}}} & $\min\left(\tfrac{2}{L},\tfrac{1}{\sqrt{nK}}\right) $\\
					\bottomrule
				\end{tabular}
			\end{sc}
		\end{small}
	\end{center}
	\caption{Comparison of our results with previously known results in terms of number of functions $n$ and number of epochs $K$. For simplicity, we suppress the dependence on other problem dependent parameters such as Lipschitz constant, strong convexity, smoothness etc. These dependencies are clearly stated in Theorems~\ref{thm:main_result_2},~\ref{thm:main_result_3} and~\ref{thm:main_result_1}.}
	\label{tb:summary-table} 
	\vskip -0.1in
\end{table*}

While \sgd~holds the rare distinction of being both theoretically well understood and practically widely used, there are still significant differences between the versions of \sgd~that are studied in theory vs those used in practice. Resolving this discrepancy is an important open question. One of the major differences is that SGD is widely used in practice \emph{with out replacement} (\sgdwor). \sgdwor uses the standard \sgd update but in each epoch/pass over data, every $i\in [n]$ is sampled \emph{exactly} once but in a uniformly random position i.e., {\em without replacement}. This implies, that $\mathbb{E}_{i_t}[\nabla f(x_t; i_t)]=\nabla F(x_t)$ {\em does not} hold anymore, making the analysis of \sgdwor~significantly more challenging. 

Studies however, have shown {\em empirically} that \sgdwor~converges significantly faster than \sgd \cite{bottou2009curiously}.  
\cite{gurbuzbalaban2015random} provided the first formal guarantee for this observation and proved that the suboptimality of \sgdwor~after $K$ epochs behaves as $\order{1/K^2}$, where as the suboptimality of \sgd~is known to be $\order{1/nK}$ (and this bound is tight). Under the same assumptions,~\cite{haochen2018random} improves upon the result of~\cite{gurbuzbalaban2015random} and shows a suboptimality bound of $O\left(1/n^2K^2 + 1/K^3\right)$ where $n$ is the number of samples and $K$ is the number of epochs. However, both the above given guarantees require Hessian Lipschitz, gradient Lipschitz (also known as smoothness) and strong convexity assumptions on $F$. In contrast, \sgd's rate of $\order{\frac{1}{nK}}$ requires only strong convexity. It is also known that this rate cannot be improved with out smoothness (gradient Lipschitz). So, in this work, we ask the following question: {\em Does \sgdwor converge at a faster rate than \sgd for general smooth, strongly-convex functions (with out Hessian Lipschitz assumption)?}

We answer the above question in affirmative and show that \sgdwor can achieve convergence rate of $\otilde{1/nK^2}$ for general smooth, strongly-convex functions. Moreover, for $K\lesssim n$, our result improves upon the best known rates~\cite{haochen2018random}. Our results also improve upon the $O(1/nK)$ rate of \sgd~once $K\geq O(\kappa^2\log{nK})$ where $\kappa$ is the condition number of the problem~\eqref{eq:cond}. In contrast,~\cite{haochen2018random} requires $K\geq \order{\kappa^{1.5}\cdot \sqrt{n}}$ to improve upon the rates of \sgd. Note that in practice one takes only a few passes over the data and hence a practical method needs to demonstrate faster rate for a small number of epochs.  Finally, our analysis yields improved dependence on problem dependent parameters like $\kappa$.

As mentioned above, in many settings, we are interested in the performance of \sgdwor, when the number of passes $K$ is quite small. \cite{shamir2016without} considers an extreme version of this setting, and obtains suboptimality bounds for \sgdwor~for the \emph{first} pass, for the special case of \emph{generalized linear models}. These bounds are similar to the standard suboptimality bounds for \sgd~of $\order{1/n}$ and $\order{1/\sqrt{n}}$ for convex functions with and with out strong convexity respectively (here number of passes $K=1$). 

For the small $K$ regime, we obtain similar convergence rates of $\order{1/nK}$ and $\order{\tfrac{1}{\sqrt{nK}}}$ for smooth convex functions with and with out strong convexity respectively. This improves upon~\cite{shamir2016without} by showing the result for \emph{general} convex functions, for any number of epochs and also in terms of dependence on problem dependent parameters. These results are summarized in Table~\ref{tb:summary-table}. The first three rows of the table compare our result for {\em large} $K$ against those of \cite{gurbuzbalaban2015random} and \cite{haochen2018random}. The next two rows compare our result for {\em small} $K$ (i.e., constant $K$) against that of \cite{shamir2016without} in the presence of strong convexity. The final two rows compare our result for {\em small} $K$ against that of \cite{shamir2016without} \emph{without} strong convexity.

As noted earlier, the main challenge in analyzing \sgdwor is that in expectation, the update does not follow gradient descent (GD). That is, $\mathbb{E}_{i_t}[\nabla f(x_t; i_t)]\neq \nabla F(x_t)$.
The main proof strategy is to bound the bias in \sgdwor update, i.e., $\|\mathbb{E}_{i_t}[\nabla f(x_t; i_t)]- \nabla F(x_t)\|$ as well as the variance associated with the update, i.e., $\mathbb{E}_{i_t}[\|\nabla f(x_t; i_t)\|^2]-\|\nabla F(x_t)\|^2$. To bound the bias term,
we use a novel coupling technique for limiting Wasserstein distance between the paths of \sgdwor and \sgd. For the variance term, we use smoothness of the function to show that compared to \sgd, \sgdwor naturally leads to variance reduction. We put together these two terms and analyze them in different settings of $K$ (constant vs condition number dependent $K$) to obtain our final results (Theorems~\ref{thm:main_result_2},~\ref{thm:main_result_3} and~\ref{thm:main_result_1}). 



\textbf{Organization}: We introduce problem setup, notations, and a brief overview of related works in Section~\ref{sec:prob}. In Section~\ref{sec:results}, we present our main results, compare it with existing work and give a rough outline of our proof strategy. In Section~\ref{sec:coupling_and_wasserstein}, we introduce coupling and Wasserstein distances and use these ideas to state and prove some important lemmas in our context. Section~\ref{sec:proofs} presents the proofs of our main results. Finally, we conclude with Section~\ref{sec:conc}. Due to space limitations, some of the proofs are presented in the appendix.
\section{Problem Setup}
\label{sec:prob}
\renewcommand{\algorithmicrequire}{\textbf{Input:}}
\renewcommand{\algorithmicensure}{\textbf{Output:}}

\begin{figure*}[t]
	\begin{minipage}{.49\textwidth}
		\begin{algorithm}[H]
			\small   
			\caption{\sgd: SGD with replacement}\label{alg:sgd}
			\begin{algorithmic}[1]
				\Require Functions $f(x;i), i \in [n]$, convex set $\mathcal{W}$, maximum number of epochs $K$, step-size sequence $\alpha_{k,i},\ k\in [K],\ i\in [n]$
				\State $x_n^0\leftarrow 0$
				\For{$k\in[K]$}
				\State $x_0^k\leftarrow x_n^{k-1}$
				\For{$0\leq i\leq i\leq n-1$}
					\State $j^k_i\leftarrow Unif[n]$
					\State $x_{i+1}^k\leftarrow \Pi_{\mathcal{W}}\left(x_i^k - \alpha_{k,i}\nabla f\left(x_i^{k};j^k_i\right)\right)$
				\EndFor 
				\EndFor 
			\end{algorithmic}
		\end{algorithm}
	\end{minipage}\quad
	\begin{minipage}{.49\textwidth}
	\begin{algorithm}[H]
		\small    
		\caption{\sgdwor: SGD without replacement}\label{alg:sgdwor}
		\begin{algorithmic}[1]
				\Require Functions $f(x;i), i \in [n]$, convex set $\mathcal{W}$, number of epochs $K$, step-size sequence $\alpha_{k,i},\ k\in [K],\ i\in [n]$
				\State $x_n^0\leftarrow 0$
				\For{$k\in[K]$}
					\State $x_0^k\leftarrow x_n^{k-1}$
					\State $\sigma_k \leftarrow $ uniformly random permutation of $[n]$
					\For{$0\leq i\leq i\leq n-1$}
						\State $x_{i+1}^k\leftarrow \Pi_{\mathcal{W}}\left(x_i^k - \alpha_{k,i}\nabla f\left(x_i^{k};\sigma_k(i+1)\right)\right)$
					\EndFor
				\EndFor
		\end{algorithmic}
	\end{algorithm}
	\end{minipage}
\end{figure*}

Given convex functions $f(;1),\dots,f(;n) : \mathbb{R}^d \to \mathbb{R}$, we consider the following optimization problem: 
\begin{equation}\label{eq:prob}\min_{x\in \mathcal{W}} F(x) := \frac{1}{n}\sum_{i=1}^n f(x;i)\,,\end{equation}
where $\mathcal{W}\subset \mathbb{R}^d$ is a closed convex set. We will refer to $F$ as the objective function and $f(\cdot;i)$ as the component functions. Henceforth, we let $x^{*}$ denote the minimizer of $F$ over $\mathcal{W}$ and $\Pi_{\mathcal{W}}$ denote the projection operator onto the set $\mathcal{W}$. We study \sgdwor when applied to the above problem. The algorithm takes $K$ passes (epochs) over the data. In each pass, it goes through the component functions in a random order $\sigma_k:[n]\rightarrow[n]$ and requires a step size sequence $\alpha_{k,i} \geq 0$ for $k\in [K],\ 0\leq i\leq n-1$ for computing stochastic gradient. See Algorithm~\ref{alg:sgdwor} for pseudo-code. 

For simplicity of analysis and exposition, we assume constant step-sizes $\alpha_{k,i}$. 
For our analysis, we assume that the component functions are twice differentiable, uniformly $G$ lipschitz and $L$ smooth over $\mathcal{W}$.
\begin{assumption}[Lipschitz Continuity]\label{as:lc}There exists $G > 0$ such that 
$\|\nabla f(x;i)\| \leq G \; \forall \; x \in \mathcal{W}$ and $i \in [n]$. 
\end{assumption}

\begin{assumption}[Smoothness/Gradient Lipschitz]\label{as:smooth} There exists $L > 0$ such that, $\|\nabla f(x;i) - \nabla f(y;i)\| \leq L\|x-y\| \; \forall \; x,y\in \mathcal{W}$ and $i \in [n]$.
\end{assumption}
In addition, we require strong-convexity of $F(\cdot)$ for Theorem~\ref{thm:main_result_2} and Theorem~\ref{thm:main_result_3} to hold. 
\begin{assumption}[Strongly-convex]\label{as:sc} There exists $\mu > 0$ s. t.
$F(y) \geq F(x) + \langle \nabla F(x), y-x\rangle + \frac{\mu}{2} \|y-x\|^2 \; \forall \; x,y\in \mathcal{W}$.
\end{assumption}
We define condition number $\kappa$ of the problem \eqref{eq:prob} as: 
\begin{equation}
\label{eq:cond}
\kappa=L/\mu,
\end{equation}
where $L$ and $\mu$ are smoothness and strong convexity parameters defined by Assumptions ~\ref{as:smooth} and \ref{as:sc}, respectively. Finally, we denote the distance of initial point $x_i^0$ from the optimum by $D$ i.e., $D \defeq \norm{x_i^0 - x^*}$.

\subsection{Related Work}
Gradient descent (GD) and it's variants are well-studied in literature \cite{bubeck2015convex}. If Assumption~\ref{as:lc} is satisfied, then suboptimality of GD (more precisely subgradient descent) with averaging is bounded by $O(G\cdot D/\sqrt{K})$ where $K$ is the number of GD iterations. With Assumption \ref{as:smooth}, the convergence rate improves to $O(LD^2/K)$ and with additional Assumption~\ref{as:sc}, it further improves to $O(e^{-K/\kappa}LD^2)$ where $\kappa$ is defined by \eqref{eq:cond}. For smooth functions, accelerated gradient descent (AGD) further improves the rates to $O(LD^2/K^2)$ and $O(e^{-K/\sqrt{\kappa}}LD^2)$, in the above two settings respectively~\cite{bubeck2015convex}. 

Each iteration of GD requires a full pass over data and hence requires prohibitively large $O(n\cdot T_f)$ computation where $T_f$ is the computation cost of evaluating gradient of any $f(x;i)$ at any $x$. In contrast, SGD (Algorithm~\ref{alg:sgd}) requires only $O(T_f)$ computation per step. Moreover, \sgd's suboptimality after $K$ passes over the data is $O(G\cdot D/\sqrt{nK})$ with Assumption~\ref{as:lc}. Similarly, it is $O(G^2/\mu \cdot 1/(nK))$ if Assumption~\ref{as:sc} also holds. Without any additional assumptions, these rates are known to be tight.

With additional Assumption~\ref{as:smooth}, people have designed acceleration methods for \sgd such as SAGA~\cite{defazio2014saga}, SVRG~\cite{johnson2013accelerating}, SDCA~\cite{shalev2013stochastic} and SAG~\cite{schmidt2017minimizing} - these methods achieve variance reduction using previous iterates in the algorithm and obtain faster rates of convergence. Note that none of these results applies for \sgdwor as sampling without replacement introduces dependencies between iterates and gradients. But, at a high-level, our result shows that \sgdwor~naturally achieves some amount of variance reduction giving better convergence rate than \sgd.

There have also been other works that study \sgdwor.~\cite{recht2012beneath} relate the performance of \sgdwor to a noncommutative version of arithmetic-geometric mean inequality~\cite{zhang2014note,israel2016arithmetic}. However, this conjecture has not yet been fully resolved.~\cite{ying2018stochastic} shows that for a small enough fixed step size, the distibution of \sgdwor converges closer to the optimum than \sgd.

\section{Main Results}
\label{sec:results}
In this section, we present our main results for~\sgdwor~and the main ideas behind the proofs. Recall that $x_i^k$ denotes the iterates of \sgdwor~ and let $x^*$ be a minimizer of $F(\cdot)$ over $\mathcal{W}$. We define $d_{i,k} := \|x_i^k - x^*\|$. We now present our first result that improves upon the convergence rate of \sgd~for large $K$.
%
%
%
\begin{theorem}
\label{thm:main_result_2} Suppose $F(\cdot)$ satisfies Assumptions~\ref{as:lc}-\ref{as:sc}. Fix $l > 0$ and let number of epochs $K > 32 l \kappa^2\log{nK}$. Let $x_{i}^{k}$ be the iterates of \sgdwor(Algorithm~\ref{alg:sgdwor}) when applied to $F(\cdot)$ with constant learning rate $\alpha_{k,i} = \alpha \defeq 4l \frac{\log{nK}}{\mu nK}$. Then the following holds for the tail average $\hat{x} \defeq \frac{1}{K-\lceil\tfrac{K}{2}\rceil + 1} \sum_{k= \lceil\tfrac{K}{2}\rceil}^{K} x_{0}^{k}$ of the iterates: 
	
	
	
\begin{align*}
\mathbb{E}[F(\hat{x})] - F(x^*) &\leq  O\left(\mu\tfrac{d_{0,1}^2}{(nK)^l}\right) + O\left(\tfrac{\kappa^2G^2}{\mu}\tfrac{(\log{nK})^2}{nK^2}\right).
\end{align*}
\end{theorem}
\textbf{Remarks}:
\begin{itemize}
	\item The error has two terms -- first term depending on initial error $d_{0,1}$ and second term depending on problem parameters $L, G$ and $\mu$. The dependence on initial error can be made to decay very fast by choosing $l$ to be a large enough constant, i.e., $K=\Omega(\kappa^2)$. In this case, the leading order term is the second term which decays as $\order{\frac{1}{nK^2}}$. Our result  improves upon the $\order{\frac{G^2}{\mu nK}}$ rate of \sgd~once $K > \order{\kappa^2\log{nK}}$.
	\item Our result improves upon the state of the art result for \sgdwor~by~\cite{haochen2018random} as long as $K \leq \kappa n$, which captures the most interesting setting in practice. Furthermore, we do not require the additional Hessian Lipschitz assumption. For the sake of clarity, \cite{haochen2018random} keeps all parameters other than $\mu$ constant and takes $\kappa = \Theta(1/\mu)$ to get suboptimality of $\tilde{O}\left(\frac{\kappa^4}{n^2K^2} + \frac{\kappa^4}{K^3} + \frac{\kappa^6}{K^4}\right)$. By the same token, our suboptimality is $\tilde{O}(\frac{\kappa^3}{nK^2})$.

\end{itemize}

Note that Theorem~\ref{thm:main_result_2} requires the number of passes $K > \kappa^2$. We now present results that apply even for small number of passes. In this setting, we match the rates of \sgd. The problem setting is the same as Theorem~\ref{thm:main_result_2}.

\begin{theorem}
\label{thm:main_result_3}
Suppose $F(\cdot)$ satisfies Assumptions~\ref{as:lc}-\ref{as:sc}. Let $x_{i}^{k}$ be the iterates of \sgdwor(Algorithm~\ref{alg:sgdwor}) when applied to $F(\cdot)$ with constant learning rate $\alpha_{k,i} = \alpha \defeq \min\left(\tfrac{2}{L},4l \tfrac{\log{nK}}{\mu nK}\right)$ for a fixed $l>0$. Then the following holds for the tail average $\hat{x} \defeq \frac{1}{K-\lceil\tfrac{K}{2}\rceil + 1} \sum_{k= \lceil\tfrac{K}{2}\rceil}^{K} x_{0}^{k}$ of the iterates:
\begin{align*}
 \mathbb{E}[F(\hat{x})]-F(x^{*}) 
&= O\left(\mu\tfrac{\|x_0^1-x^*\|^2}{(nK)^l} + L\tfrac{\|x_0^1-x^*\|^2}{(nK)^{(l+1)}}\right) + O\left(\tfrac{G^2\log{nK}}{\mu nK} + \tfrac{L^2G^2\log{nK}}{\mu^3 n^2K^2}\right).
\end{align*}
\end{theorem}
\textbf{Remarks}:
\begin{itemize}
	\item The dependence on initial error can be made to decay as fast as any polynomial by choosing $l$ to be a large enough constant. 
	\item Our result is the first such result for general smooth, strongly-convex functions and for arbitrary $K$; recall that the result of \cite{shamir2016without} requires $F$ to be a generalized linear function and requires $K=1$. Furthermore, even in setting of  \cite{shamir2016without}, our result improves upon best known bounds when $nK > \kappa^2$. In this case, our error rate is $\order{\frac{G^2 \log nK}{\mu n K}}$ that matches the rate of \sgd upto log factors. The result of~\cite{shamir2016without} does not obtain this rate even when $n \rightarrow \infty$.
\end{itemize}
The above two theorems require $F(\cdot)$ to be strongly convex (Assumption~\ref{as:sc}). We now present our result for $F$ that need not satisfy the strong convexity assumption.
\begin{theorem}
	\label{thm:main_result_1}
	Suppose $F(\cdot)$ satisfies Assumptions~\ref{as:lc}-\ref{as:smooth} and that $\mathsf{diam}(\mathcal{W}) \leq D$ . The average $\hat{x} \defeq \frac{\sum_{i=0}^{n-1}\sum_{k=1}^{K}x_{i}^{k}}{Kn}$ of \sgdwor~(Algorithm~\ref{alg:sgdwor}) with constant learning rate $\alpha_{k,i} = \alpha \defeq \min\left(\tfrac{2}{L},\tfrac{D}{G\sqrt{Kn}}\right)$ satisfies: 
	$$\mathbb{E}[F(\hat{x})] - F(x^{*}) \leq \tfrac{D^2L}{4nK} + \tfrac{3GD}{\sqrt{nK}}.$$ 
\end{theorem}
\textbf{Remarks}:
\begin{itemize}
	\item The second term of $\order{\frac{GD}{\sqrt{nK}}}$ is the same as the rate of \sgd in this setting. This becomes the leading order term once $nK \geq \frac{L^2 D^2}{G^2}$.
	\item Our result is the first such result for general smooth, Lipschitz convex functions. The earlier result by~\cite{shamir2016without} applied only for generalized linear models but does not require smoothness assumption.
\end{itemize}

\subsection{Necessity of Smoothness}
In the classical analysis of \sgd for $\order{\frac{1}{nK}}$ rate, one only requires Assumptions~\ref{as:lc} and~\ref{as:sc}. In this section, we outline an argument showing that obtaining a better rate than $\order{\frac{1}{nK}}$ for \sgdwor as in Theorem~\ref{thm:main_result_2}, requires additional Assumption~\ref{as:smooth} (smoothness). In contrast, it is well known that the rate of $\order{\frac{1}{nK}}$ is tight for \sgd even with additional Assumption~\ref{as:smooth}.
Consider the example where all the component functions are same. i.e, $f(x;i) = g(x)$ for all $1\leq i \leq n$ and $x\in \mathbb{R}^d$. Then, running \sgdwor~ for optimizing $F(x) := \frac{1}{n}\sum_{i=1}^{n}f(x;i) = g(x)$  for $K$ epochs (over a closed convex set $\mathcal{W}$) is the same as running gradient descent over $F(x)$ for $nK$ iterations.


Given any $T = nK$, \cite[Theorem~3.13]{bubeck2015convex} shows the existence of a function satisfying Assumptions~\ref{as:lc} and~\ref{as:sc} and a closed convex set $\mathcal{W}$ such that the suboptimality of all iterates up to the $T^{\textrm{th}}$ iteration of GD--hence, for all the iterates up to $K^{\textrm{th}}$ epoch of \sgdwor--is lower bounded by $\frac{G^2}{8\mu nK}$. This establishes the necessity of Assumption~\ref{as:smooth} for obtaining improved rates over \sgd as in Theorem~\ref{thm:main_result_2}.
\subsection{Proof Strategy}
As a general note, in the proofs, we assume that $\mathcal{W}=\R^d$, which avoids the projection operator $\Pi_{\mathcal{W}}$. All the steps go through in a straight forward fashion even with this projection operator.
When we try to apply the classical proof of rate of convergence of \sgd to \sgdwor, the major problem we encounter is that $\mathbb{E} [f(x_i^{k};\sigma_k(i+1))] \neq \mathbb{E}[F(x_{i}^{k})]$. In section~\ref{sec:coupling_and_wasserstein}, we propose a coupling sequence and use it to bound a certain Wasserstein distance to argue that $\mathbb{E}[ f(x_i^{k};\sigma_k(i+1)) ]\approx \mathbb{E}[F(x_{i}^{k})]$. This along with standard analysis tools then yields Theorems~\ref{thm:main_result_3} and~\ref{thm:main_result_1}.

However, this technique does not suffice to obtain faster rate as in Theorem~\ref{thm:main_result_2}. So, to prove Theorem~\ref{thm:main_result_2}, we show that in expectation, \sgdwor~over one epoch approximates one step of GD applied to  $F$. Therefore, $K$ epochs of \sgdwor~ approximates GD iterates after $K$ iterations. Recall  
$$x_{0}^{k+1} = x_{0}^{k} - \alpha_k \sum_{i=0}^{n-1}\nabla f(x_i^{k},\sigma_k(i+1)).$$

If $x_i^k \approx x_0^k$, then the equation above implies: $$x_{0}^{k+1} \approx x_{0}^{k} - \alpha_k \sum_{i=0}^{n-1}\nabla f(x_0^{k},\sigma_k(i+1)) = x_0^k - n\alpha_k \nabla F(x_0^k)\,.$$ We observe that the right hand side is one step of gradient descent. Lemma~\ref{lem:temporal_regularity} in Section~\ref{sec:coupling_and_wasserstein} makes this argument rigorous as it shows that $\mathbb{E}[\|x_i^k - x_0^k\|^2]$ becomes small as $F(x_0^k) \rightarrow F(x^*)$.

\section{Coupling and Wasserstein distance}
\label{sec:coupling_and_wasserstein}
In this section, we develop the required machinery to show: 
\begin{equation}
\mathbb{E}[f(x_i^{k};\sigma_k(i+1))] \approx \mathbb{E}[F(x_{i}^{k})].
\label{eq:approx_independent}
\end{equation}
Define the following exchangeable pair: suppose we run the algorithm for $k-1$ epochs using permutations $\sigma_1,\dots,\sigma_{k-1}$ to obtain $x_0^{k}$. When $k=1$, this means that we start with the same starting point $x_0^1$. We draw two independent uniform permutations: $\sigma_k$ and $\sigma_k^{\prime}$. If we run the $k$-th epoch with permutation $\sigma_k$, we denote the $k$-th epoch iterates by $\left(x_i(\sigma_k)\right)_{i=1}^{n}$ to explicity show the dependence on $\sigma_k$. Similarly, the sequence obtained by using the permutation $\sigma^{\prime}_k$ for the $k$-th epoch is denoted by $\left(x_i(\sigma_k^{\prime})\right)_{i=1}^{n})$. It is clear that $(x_i(\sigma_k^{\prime}))_{i=1}^{n}$ is independent and indentically distributed as $(x_i(\sigma_k))_{i=1}^{n}$. We note:
{\small \begin{equation}\label{eq:eqv}\mathbb{E}[f(x_i(\sigma^{\prime}_k);\sigma_k(i+1))] = \mathbb{E}[f(x_i(\sigma_k^{\prime}))] = \mathbb{E}[F(x_i^k)]\,.\end{equation}}
Here the first equality follows from the fact that $\sigma_k$ is independent of $\sigma_k^{\prime}$ (and applying Fubini's theorem). The second equality follows from the fact that $x_i(\sigma_k^{\prime})$ and $x_i(\sigma_k) $ are identically distributed.  Therefore, to show~\eqref{eq:approx_independent}, we need to show that:
$\mathbb{E}[f(x_i(\sigma^{\prime}_k);\sigma_k(i+1))] - \mathbb{E}[f(x_i(\sigma_k);\sigma_k(i+1))] \approx 0\,.$
Since $f(\cdot;j)$ is uniformly lipschitz, a bound on the Wasserstein distance between $x_i(\sigma_k)$ and $x_i(\sigma_k^{\prime})$ would imply the above result. That is, Lemma~\ref{lem:wasserstein_connection} shows that $ \mathbb{E}[f(x_i(\sigma^{\prime}_k);\sigma_k(i+1))] - \mathbb{E}[f(x_i(\sigma_k);\sigma_k(i+1))]$ is bounded by the Wasserstein distance between $x_i(\sigma_k)$ and $x_i(\sigma_k^{\prime})$, and Lemma~\ref{lem:wasserstein_stability_bound} then bounds the Wasserstein distance, to bound the above quantity.

 We first introduce some notation to prove the result. Let $\mathcal{D}_{i,k} := \mathcal{L}(x_i(\sigma_k))$ and $\mathcal{D}_{i,k}^{(r)} := \mathcal{L}\left(x_i(\sigma_k)|\sigma_k(i+1) = r\right)$. Here $\mathcal{L}(X)$ denotes the distribution of the random variable $X$. We let $\mathsf{Lip}_d(\beta)$ be the set of all $\beta$ lipschitz functions from $\mathbb{R}^d \to \mathbb{R}$. 
\begin{definition}
\label{def:wasserstein_distance}
Let $P$ and $Q$ be two probability measures over $\mathbb{R}^d$ such that $\mathbb{E}_{X \sim P}[\|X\|^2] < \infty$ and $\mathbb{E}_{Y \sim Q}[\|Y\|^2] < \infty$.  Let $X \sim P$ and $Y \sim Q$ be random vectors defined on a common measure space (i.e, they are coupled). We define Wasserstein-1 and Wasserstein-2 distances between $P $ and $Q$ as:
\begin{align*}
	\dw (P,Q) &\defeq \inf_{\substack{(X,Y) : \\ X\sim P \\ Y \sim Q}}\mathbb{E}[\|X-Y\|] \; \mbox{, and}\\
	\dwtwo (P,Q) &\defeq \inf_{\substack{(X,Y) : \\ X\sim P \\ Y \sim Q}}\sqrt{\mathbb{E}[\|X-Y\|^2] },
\end{align*}
respectively. Here the infimum is over all joint distributions over $(X,Y)$ with prescribed marginals.
\end{definition}
By Jensen's inequality, we have $\dwtwo(P,Q) \geq \dw(P,Q)$.
The following result gives a fundamental characterization of Wasserstein distance~\cite{santambrogio2015optimal}.
\begin{theorem}[Kantorovich Duality]
\label{thm:kantorovich_duality}
 Let $P$ and $Q$ satisfy the conditions in Definition~\ref{def:wasserstein_distance}. Let $X \sim P$ and $Y \sim Q$ then:
$$\dw(P,Q) = \sup_{g \in \mathsf{Lip}_d(1) } \mathbb{E}[g(X)] - \mathbb{E}[g(Y)].$$
\end{theorem}
We can use Theorem~\ref{thm:kantorovich_duality} to bound the approximation error in~\eqref{eq:approx_independent} in terms of average Wasserstein-1 distance between $\mathcal{D}_{i,k}$ and $\mathcal{D}_{i,k}^{(r)}$.
\begin{lemma}
\label{lem:wasserstein_connection}
$$\bigr|\mathbb{E}[F(x_i^k)] - \mathbb{E}[f(x_i^k;\sigma_k(i+1))]\bigr| \leq \tfrac{G}{n}\sum_{r=1}^{n} \dw\left(\mathcal{D}_{i,k},\mathcal{D}_{i,k}^{(r)}\right)$$
\end{lemma}

\begin{proof} Let $R_j := \sigma_k(j)$ for all $j \in [n]$. Using \eqref{eq:eqv}: 
\begin{align*}
\bigr|\mathbb{E}[F(x_i^k)] - \mathbb{E}[f(x_i^k;R_{i+1})]\bigr| &= \biggr|\mathbb{E}[f(x_i(\sigma^{\prime}_k);R_{i+1})] - \mathbb{E}[f(x_i(\sigma_k);R_{i+1})]\biggr| \\
&\leq  \tfrac{1}{n}\sum_{r=1}^{n}\bigr|\mathbb{E}\left[ f(x_i(\sigma^{\prime}_k);r)\right] - \mathbb{E} \left[f(x_i(\sigma_k);r)\bigr|R_{i+1}=r\right]\bigr|\\
&\leq \tfrac{1}{n}\sum_{r=1}^{n}\sup_{g\in \mathsf{Lip}_d(G)}\biggr(\mathbb{E}\left[ g(x_i(\sigma^{\prime}_k))\right] - \mathbb{E} \left[g(x_i(\sigma_k))\bigr|R_{i+1}=r\right]\biggr)\\
&= \tfrac{1}{n}\sum_{r=1}^n G\cdot \dw\left(\mathcal{D}_{i,k},\mathcal{D}_{i,k}^{(r)}\right),
\end{align*}
where the second step follows from triangle inequality and the fact that $x_i(\sigma_k^{\prime})$ is independent of $\sigma_k$. Second to last inequality follows from the fact that $f\in Lip_d(G)$ and the last inequality follows from Theorem~\ref{thm:kantorovich_duality}. We also used the fact that conditioned on $\sigma_k(i+1) = r$, $x_i^k(\sigma_k^{\prime}) \sim \mathcal{D}_{i,k}$
\end{proof}

From Lemma~\ref{lem:wasserstein_connection}, we see that we only need to upper bound $\dw\left(\mathcal{D}_{i,k},\mathcal{D}_{i,k}^{(r)}\right)$. We hope to use the definition of Wasserstein-1 distance (Definition~\ref{def:wasserstein_distance}) by constructing a nice coupling between $\mathcal{D}_{i,k}$ and $\mathcal{D}_{i,k}^{(r)}$, that we present in the following lemma; see Appendix~\ref{app:lem} for a proof of the lemma. 
\begin{lemma}
	Given $k$, suppose $\alpha_{k,i}$ is a non-increasing function of $i$ and $\alpha_{k,0} \leq \frac{2}{L}$. Then almost surely, 
	$\forall \ i \ \in \ [n]$, 
	\begin{equation}
	\|x_i(\sigma'_k)-x_i(\sigma_{k})\| \leq 2 G \alpha_{k,0} \cdot \left| \left\{j \leq i : \sigma_{k}(j) \neq \sigma'_{k}(j) \right\} \right|. \label{eq:strong_stability}
	\end{equation}
	Here $\left| \left\{j \leq i : \sigma_{k}(j) \neq \sigma'_{k}(j) \right\} \right|$ is the number of iterations till $i$ where the two permutations $\sigma_k$ and $\sigma'_k$ choose different component functions.
	\label{lem:stability}
\end{lemma}

A key ingredient in the proof of the above lemma is the following standard result which says that gradient step with small enough step size is contracting for smooth convex functions. 
\begin{lemma}
If $g$ is convex and $\|\nabla^2 g\| \leq L$, then,
$$\|\nabla g(x) - \nabla g(y)\|^2 \leq L \langle \nabla g(x) -\nabla g(y),x-y\rangle.$$
\label{lem:smooth_convex_property}
\end{lemma}
See~\cite[Theorem 2.1.5]{nesterov2013introductory} for a proof.

\subsection{Coupling $\sigma_{k}$ and $\sigma'_{k}$}
\label{subsec:wasserstein_coupling_bound}
In this section, we construct a coupling between $\sigma_{k}$ and $\sigma'_{k}$ that minimizes the bound in Lemma~\ref{lem:stability}.
Let $\perm_n$ be the set of all permutations over $n$ letters. For $a,b \in [n]$, we define the exchange function $E_{a,b} : \perm_n \to \perm_n$: for any $\tau \in \perm_n$, $E_{a,b}(\tau) $ gives a new permutation where $a$-th and $b$-th entries of $\tau$ are exchanged and it keeps everything else same. We construct the operator $\Lambda_{r,i} : \perm_n \to \perm_n$:
$$\Lambda_{r,i} (\tau) = \begin{cases}
\tau &\qquad \text{if } \tau(i+1)= r
\\ E_{i+1,j}(\tau) &\qquad \text{if } \tau(j) = r \text{ and } j \neq i+1
\end{cases}$$
Basically, $\Lambda_{r,i}$ makes a single swap so that $i+1$-th position of the permutation is $r$. Clearly, if $\sigma_k$ is a uniformly random permutation, then $\Lambda_{r,i}(\sigma_k)$ has the same distribution as $\sigma_k | \sigma_k(i+1)  = r$. We use the coupling characterization of $\dw$ to conclude:

\begin{lemma}
\label{lem:wasserstein_stability_bound}
Let $k$ be fixed. When $\alpha_{k,0} \leq \frac{2}{L}$ and $\alpha_{k,i}$ be a non-increasing function of $i$,
$$\dw\left(\mathcal{D}_{i,k},\mathcal{D}_{i,k}^{(r)}\right)\leq \dwtwo\left(\mathcal{D}_{i,k},\mathcal{D}_{i,k}^{(r)}\right) \leq 2\alpha_{k,0}G,$$
where $\mathcal{D}_{i,k},$ $\mathcal{D}_{i,k}^{(r)}$ are defined above. Consequently, from Lemma~\ref{lem:wasserstein_connection}, we conclude:
$$\bigr|\mathbb{E}[F(x_i^k)] - \mathbb{E}[f(x_i^k;\sigma_k(i+1))]\bigr| \leq 2\alpha_{k,0}G^2.$$
\end{lemma}

\begin{proof}
Let $\sigma_k$ be a uniformly random permutation and $r \in [n]$. Therefore, $x_i(\sigma_k) \sim \mathcal{D}_{i,k}$ and $x_i\left(\Lambda_{r,i}(r)\right) \sim \mathcal{D}_{i,k}^{(r)}$. This gives a coupling between $\mathcal{D}_{i,k}$ and  $\mathcal{D}_{i,k}^{(r)}$. By definition of Wasserstein distance:

\begin{equation}
 \dwtwo\left(\mathcal{D}_{i,k},\mathcal{D}_{i,k}^{(r)}\right) \leq \sqrt{\mathbb{E}\|x_i(\sigma_k) - x_i(\Lambda_{r,i}\sigma_k) \|^2}
 \label{eq:coupling_bound_in_action}
 \end{equation}

It is clear that $\bigr|\{j \leq i : \sigma_k(j) \neq \left[\Lambda_{r,i}(\sigma_k)\right](j) \} \bigr|\leq 1$ almost surely. Therefore, from Lemma~\ref{lem:stability}, we conclude that $\|x_i(\sigma_k) - x_i(\Lambda_{r,i}\sigma_k) \| \leq 2\alpha_{k,0}G$ almost surely. Together with Equation~\ref{eq:coupling_bound_in_action}, and the fact that $\dwtwo \geq \dw$ we conclude the result.
\end{proof}
The lemmas presented above tightly bound the difference in suboptimality between iterates of \sgd and \sgdwor. These will be used in proving Theorems~\ref{thm:main_result_3} and~\ref{thm:main_result_1}, matching the rates of \sgd. For Theorem~\ref{thm:main_result_2}, we need to show that there is some amount of automatic variance reduction while running \sgdwor. In order to do this, we need to show that the iterates $x_i^k$ do not move much when they are close to the optimum. The following lemma makes this precise.
\begin{lemma}
\label{lem:temporal_regularity}
Let $\alpha_{k,0} < \frac{2}{L}$ and $\alpha_{k,j}$ be a non-increasing sequence in $j$ for a given $k$. For any $i \in [n]$, we have:
\begin{align*}
	\mathbb{E}[\|x_i^k - x_0^k\|^2] &\leq 5i\alpha_{k,0}^2G^2 + 2i\alpha_{k,0}\cdot \mathbb{E}\left[F(x_0^k) - F(x^{*})\right],\\\mbox{ and }\ 
	\mathbb{E}[d_{i,k}^2] &\leq \mathbb{E}[d_{0,k}^2] + 5i\alpha_{k,0}^2\cdot G^2,
\end{align*}
where we recall that $d_{i,k} \defeq \|x_i^k - x^*\|$.
\end{lemma}
See Appendix~\ref{app:lem} for a detailed proof of the above lemma. 
\section{Proofs of Main Results}
\label{sec:proofs}
In this section, we will present proofs of Theorems~\ref{thm:main_result_2},~\ref{thm:main_result_3} and~\ref{thm:main_result_1} using the results from the previous section.
\subsection{Proof of Theorem~\ref{thm:main_result_2}}
In this subsection, for the sake of clarity of notation, we take $R_j \defeq \sigma_k(j)$ for every $ j\in [n]$.
Recall the definition $d_{i,k}:= \|x_i^k-x^{*}\|$. From the definition of~\sgdwor, and the choice of step sizes $\alpha_{k,i} = \alpha = 4l \frac{\log{nK}}{\mu nK}$, we have:

$$x_{0}^{k+1} = x_{0}^{k} - \alpha \sum_{i=0}^{n-1}\nabla f(x_i^{k},R_{i+1}).$$

Using the hypothesis that $\alpha \leq \frac{2}{L}$ (since $\tfrac{\mu}{L} \leq 1$) and taking norm squared on both sides,
\begin{align}
d_{0,k+1}^2  &= d_{0,k}^2 - 2\alpha\sum_{i=0}^{n-1}\langle \nabla f(x_i^{k},R_{i+1}), x_{0}^{k}-x^{*}\rangle + \alpha^2 \biggr\| \sum_{i=0}^{n-1}\nabla f(x_i^{k},R_{i+1})\biggr\|^2 \nonumber\\
&= d_{0,k}^2 -2n\alpha\langle \nabla F(x_{0}^k),x_{0}^k -x^{*}\rangle - 2\alpha\sum_{i=0}^{n-1}\langle \nabla f(x_i^{k},R_{i+1}) -\nabla F(x_0), x_{0}^{k}-x^{*}\rangle \nonumber \\&\quad + \alpha^2\biggr \| \sum_{i=0}^{n-1}\nabla f(x_i^{k},R_{i+1})\biggr\|^2\nonumber \\
&\leq d_{0,k}^2(1-n\alpha \mu) - 2n\alpha(F(x_{0}^{k}) - F(x^{*}))  -2\alpha\sum_{i=0}^{n-1}\langle \nabla f(x_i^{k},R_{i+1}) -\nabla F(x_0), x_{0}^{k}-x^{*}\rangle \nonumber \\&\quad + \alpha^2\biggr\| \sum_{i=0}^{n-1}\nabla f(x_i^{k},R_{i+1})\biggr\|^2, \label{eq:distance_recursion}
\end{align}
where we used strong convexity of $F$ in the third step. We consider the term:
\begin{align*}
&T_1  \defeq -2\alpha\sum_{i=0}^{n-1}\langle \nabla f(x_i^{k},R_{i+1}) -\nabla F(x_0), x_{0}^{k}-x^{*}\rangle \\ &= -2\alpha \sum_{i=0}^{n-1}\langle\nabla f(x_i^{k};R_{i+1})-\nabla f(x_0^{k};R_{i+1}), x_0^k-x^{*}\rangle.
\end{align*}
Taking expectation, we have
\begin{align}
&\mathbb{E}[T_1] \nonumber\\ &= -2\alpha \mathbb{E}\sum_{i=0}^{n-1}\langle\nabla f(x_i^{k};R_{i+1})-\nabla f(x_0^{k};R_{i+1}), x_0^k-x^{*}\rangle \nonumber \\
&\leq 2\alpha L\sum_{i=0}^{n-1}\mathbb{E}\left[\|x_i^k-x_0^k\|.\|x_0^k-x^{*}\|\right] \nonumber\\ &\leq  2\alpha L\sum_{i=0}^{n-1}\sqrt{\mathbb{E}\|x_i^k-x_0^k\|^2}\sqrt{\mathbb{E}\|x_0^k-x^{*}\|^2}\nonumber \\
&\leq 2\alpha L n \sqrt{\mathbb{E}[d_{0,k}^2]}\sqrt{5n\alpha^2G^2 + 2n\alpha\mathbb{E}\left[F(x_0^k) - F(x^*)\right]},
 \label{eq:expectation_bound_1}
\end{align}
where we used Cauchy-Schwarz and smoothness in the second step and Lemma~\ref{lem:temporal_regularity} in the last step.
We now consider $$T_2 \defeq \alpha^2 \biggr\| \sum_{i=0}^{n-1}\nabla f(x_i^{k};R_{i+1})\biggr\|^2\,.$$
We use the fact that: $\nabla F(x^{*}) = 0 = \sum_{i=0}^{n-1} \nabla f(x^{*};R_{i+1})$ in the equation above to conclude:
\begin{align*}
T_2 &= \alpha^2 \biggr\| \sum_{i=0}^{n-1}\nabla f(x_i^{k};R_{i+1}) - \nabla f(x^{*};R_{i+1})\biggr\|^2 \\
&\leq \alpha^2 \left[\sum_{i=0}^{n-1}\biggr\|\nabla f(x_i^{k};R_{i+1}) - \nabla f(x^{*};R_{i+1})\biggr\|\right]^2 \\
&\leq \alpha^2 L^2\left[\sum_{i=0}^{n-1}\bigr\|x_i^{k} - x^{*}\bigr\|\right]^2 = \alpha^2 L^2\sum_{i=1}^{n}\sum_{j=1}^{n} d_{i,k}d_{j,k}.
\end{align*}
Taking expectation, we have
\begin{align}
\mathbb{E}[T_2] &\leq \alpha^2 L^2 \sum_{i=1}^{n}\sum_{j=1}^{n}\mathbb{E}[d_{i,k}d_{j,k}] \nonumber \\
&\leq \alpha^2 L^2 \sum_{i=1}^{n} \sum_{j=1}^{n} \sqrt{\mathbb{E} [d_{i,k}^2]}\sqrt{\mathbb{E}[d_{j,k}^2]} \nonumber \\
&\leq \alpha^2L^2n^2\left[ \mathbb{E}[d_{0,k}^2] + 5n\alpha^2 G^2\right], \label{eq:expectation_bound_2}
\end{align}
where we again used Cauchy-Schwarz inequality and Lemma~\ref{lem:temporal_regularity}.
Applying AM-GM inequality on~\eqref{eq:expectation_bound_1}, we have:
\begin{align*}
&\mathbb{E}[T_1] \leq \alpha L n \left[\tfrac{ \mu \mathbb{E}[d_{0,k}^2]}{4L}+ \tfrac{4L\left(5n\alpha^2G^2 + 2n\alpha\mathbb{E}\left[F(x_0^k) - F(x^*)\right]\right)}{\mu}\right] \\
&= \tfrac{\alpha \mu n}{4}\mathbb{E}[d_{0,k}^2] + \tfrac{20L^2\alpha^3n^2G^2}{\mu} + \tfrac{8\alpha^2L^2n^2}{\mu}\mathbb{E}\left[F(x_0^k) - F(x^*)\right]
\end{align*}
Plugging the inequality above and~\eqref{eq:expectation_bound_2} in~\eqref{eq:distance_recursion}, we conclude:
\begin{align}
\mathbb{E}[d_{0,k+1}^2]  &\leq \mathbb{E}[d_{0,k}^2](1-n\alpha \mu) - 2n\alpha\mathbb{E}(F(x_{0}^{k}) - F(x^{*}))  +  \tfrac{\alpha \mu n}{4}\mathbb{E}[d_{0,k}^2]+ \tfrac{20L^2\alpha^3n^2G^2}{\mu} + \tfrac{8\alpha^2L^2n^2\mathbb{E}\left[F(x_0^k) - F(x^*)\right]}{\mu} \nonumber \\&\quad+ \alpha^2L^2 n^2 \mathbb{E}[d_{0,k}^2] + 5\alpha^4L^2G^2n^3\nonumber \\
&\leq \mathbb{E}[d_{0,k}^2](1-\tfrac{3n\alpha \mu}{4}+ \alpha^2n^2L^2)  - 2n\alpha\left(1-\tfrac{4\alpha n L^2}{\mu}\right)\mathbb{E}\left[F(x_{0}^{k}) - F(x^{*})\right]  + \tfrac{20L^2\alpha^3n^2G^2}{\mu}  + 5\alpha^4L^2G^2n^3.
\label{eq:distance_recursion_specialized}
\end{align}

By our choice of step sizes, it is clear that $1-\tfrac{3n\alpha\mu}{4}+ \alpha^2n^2L^2 \leq 1- \frac{n\alpha\mu}{2}$. In~\eqref{eq:distance_recursion_specialized}, we use the fact that $F(x_0^k)\geq F(x^*)$ to conclude:
$$\mathbb{E}[d_{0,k}^2] \leq \left(1- \tfrac{n\alpha\mu}{2}\right)\mathbb{E}[d_{0,k-1}^2] + \tfrac{20L^2\alpha^3n^2G^2}{\mu}  + 5\alpha^4L^2G^2n^3 $$

Unrolling the recursion above, we have:
\begin{align*}
\mathbb{E}[d_{0,k}^2] &\leq \left(1- \tfrac{n\alpha\mu}{2}\right)^k d_{0,1}^2 + \sum_{t=0}^{\infty}\left(1- \tfrac{n\alpha\mu}{2}\right)^t\left[\tfrac{20L^2\alpha^3n^2G^2}{\mu}  + 5\alpha^4L^2G^2n^3\right] \\
&\leq \exp\left(- \tfrac{nk\alpha\mu}{2}\right)d_{0,1}^2 + \tfrac{2}{n\alpha \mu}\left[\tfrac{20L^2\alpha^3n^2G^2}{\mu}  + 5\alpha^4L^2G^2n^3\right] \\
&\leq \exp\left(- \tfrac{nk\alpha\mu}{2}\right)d_{0,1}^2 + \tfrac{40L^2\alpha^2nG^2}{\mu^2}+\tfrac{10\alpha^3L^2G^2n^2}{\mu}
\end{align*}

Taking $\alpha = 4l \frac{\log{nK}}{\mu nK}$ and $k = \frac{K}{2}$, we have:
\begin{align*}
\mathbb{E}[d_{0,\tfrac{K}{2}}^2] &\leq \tfrac{d_{0,1}^2}{(nK)^l} + \tfrac{40L^2\alpha^2nG^2}{\mu^2}+\tfrac{10\alpha^3L^2G^2n^2}{\mu}.
\end{align*}
We now analyze the suffix averaging scheme given. Adding~\eqref{eq:distance_recursion_specialized} from $k = \frac{K}{2}$ to $k =K$, we conclude:

\begin{align*}
 n\alpha\tfrac{\sum_{k=\lceil\tfrac{K}{2}\rceil}^{K}\mathbb{E}\left[F(x_{0}^{k})- F(x^{*})\right]}{K-\lceil\tfrac{K}{2} \rceil +1}  \leq \tfrac{\mathbb{E}[d_{0,K/2}^2]}{K-\lceil\tfrac{K}{2} \rceil +1} + \tfrac{20L^2\alpha^3n^2G^2}{\mu}  + 5\alpha^4L^2G^2n^3.
\end{align*}
Here we have used the fact that $2n\alpha\left(1-\frac{4\alpha n L^2}{\mu}\right) \geq n\alpha$.
Since $F(\hat{x}) \leq \frac{\sum_{k=\lceil\tfrac{K}{2}\rceil}^{K}F(x_{0}^{k})}{K-\lceil\tfrac{K}{2} \rceil +1} $ by convexity of $F$, we have:
\begin{align*}
  \mathbb{E}\left[F(\hat{x})- F(x^{*})\right]  &\leq \tfrac{2}{nK\alpha}\left[ \tfrac{d_{0,1}^2}{(nK)^l}\right] + \left[\tfrac{80L^2\alpha G^2}{\mu^2 K}+\tfrac{20\alpha^2L^2G^2n}{\mu K}\right] + \tfrac{20L^2\alpha^2nG^2}{\mu}  + 5\alpha^3L^2G^2n^2 \\
 &= O\left(\mu\tfrac{d_{0,1}^2}{(nK)^l}\right) + O\left(\tfrac{\kappa^2G^2}{\mu}\tfrac{(\log{nK})^2}{nK^2}\right).
\end{align*}
This proves the theorem.

\subsection{Proof of Theorem~\ref{thm:main_result_1}}
We note that we have taken $\alpha_{k,i} = \alpha < \frac{2}{L}$.
By definition,
$x_{i+1}^{k} = \Pi_{\mathcal{W}}\left( x_{i}^{k} - \alpha \nabla f(x_i^{k};\sigma_k(i+1))\right)$.  We take $r = \sigma_k(i+1)$ below. Taking norm squared and using Lemma~\ref{lem:convex_contraction}
\begin{align}
\|x_{i+1}^{k}-x^{*}\|^2 &\leq \|x_i^k-x^{*}\|^2 +\alpha^2\|\nabla f(x^k_i;r)\|^2  -2\alpha
\langle \nabla f(x_i;r),x_i^k-x^{*}\rangle\nonumber \\
&\leq \|x_i^k-x^{*}\|^2 + \alpha^2G^2 - 2\alpha (f(x^k_i;r)-f(x^{*};r)) \nonumber
\end{align}
In the last step we have used convexity of of $f(;j)$ and the fact that $\|\nabla f(;j)\| \leq G$. Taking expectation above, and noting that $\mathbb{E}f(x^{*};\sigma_k(i+1)) = F(x^*)$ and using Lemma~\ref{lem:wasserstein_stability_bound}, we have:
\begin{align}
\mathbb{E}\|x_{i+1}^{k}-x^{*}\|^2 &\leq \mathbb{E}\|x_i^k-x^{*}\|^2 - 2\alpha \mathbb{E}(F(x^k_i)-F(x^{*})) + 5\alpha^2G^2 \label{eq:no_strong_convex_equation}
\end{align}
Summing from $i = 0$ to $n-1$ and $k=1$ to $K$, we conclude:

\begin{align*}
 \tfrac{1}{nK}\sum_{k=1}^{K}\sum_{i=0}^{n-1}(F(x^k_i)-F(x^{*})) &\leq \tfrac{\mathbb{E}\|x_{0}^{1}-x^{*}\|^2}{2\alpha nK}   + \tfrac{5}{2}\alpha G^2 \\
 &\leq  \tfrac{D^2}{2nK}\max(\tfrac{L}{2}, \tfrac{G\sqrt{nK}}{D})+ \tfrac{5G^2}{2} \min (\tfrac{2}{L}, \tfrac{D}{G\sqrt{nK}})\\
 &\leq  \tfrac{D^2}{2nK}\left(\tfrac{L}{2} + \tfrac{G\sqrt{nK}}{D}\right)+ \tfrac{5G^2}{2}. \tfrac{D}{G\sqrt{nK}}\\
 &= \tfrac{D^2L}{4nK} + \tfrac{3GD}{\sqrt{nK}}.
\end{align*}
By convexity of $F(\cdot)$, we conclude: $$F(\hat{x}) \leq  \tfrac{1}{nK}\sum_{k=1}^{K}\sum_{i=0}^{n-1}F(x^k_i)\,.$$
\section{Conclusions}\label{sec:conc}
In this paper, we study stochastic gradient descent with out replacement (\sgdwor), which is widely used in practice. When the number of passes is large, we present the first convergence result for \sgdwor, that is faster than \sgd, under standard smoothness, strong convexity and Lipschitz assumptions where as prior work uses additional Hessian Lipschitz assumption. Our convergence rates also improve upon existing results in practically interesting regimes. When the number of passes is small, we present convergence results for \sgdwor that match those of \sgd for general smooth convex functions. These are the first such results for general smooth convex functions as previous work only showed such results for generalized linear models. In order to prove these results, we use techniques from optimal transport theory to couple variants of \sgd and relate their performances. These ideas may be of independent interest in the analysis of \sgd style algorithms with some dependencies.

\bibliographystyle{unsrt}  

\bibliography{bibliography}

\appendix

\section{Supplementary Material}
\begin{lemma}
	\label{lem:convex_contraction}
	Consider $\mathbb{R}^d$ endowed with the standard inner product. For any convex set $\mathcal{W} \subset \mathbb{R}^d$ and the associated projection operator $\Pi_{\mathcal{W}}$, we have:
	
	$$\|\Pi_{\mathcal{W}}(a) - \Pi_{\mathcal{W}}(b)\| \leq \|a-b\|$$
	For all $a,b \in \mathbb{R}^d$
\end{lemma}
\begin{proof}
By Lemma 3.1.4 in \cite{nesterov2013introductory}, we conclude: 
$$\langle a - \Pi_{\mathcal{W}}(a), \Pi_{\mathcal{W}}(b)- \Pi_{\mathcal{W}}(a)\rangle \leq 0 \,.$$
Similarly, 
$$\langle  b- \Pi_{\mathcal{W}}(b),\Pi_{\mathcal{W}}(a) - \Pi_{\mathcal{W}}(b)\rangle \leq 0 \,.$$

Adding the equations above, we conclude:
$$\|\Pi_{\mathcal{W}}(a) - \Pi_{\mathcal{W}}(b)\|^2 \leq \langle a- b, \Pi_{\mathcal{W}}(a) - \Pi_{\mathcal{W}}(b)\rangle$$
Using Cauchy-Schwarz inequality on the RHS, we conclude the result.
\end{proof}

\subsection{Proof of Theorem~\ref{thm:main_result_3}}

We have chosen $\alpha_{k,i} = \alpha =  \min\left(\tfrac{2}{L},4l \tfrac{\log{nK}}{\mu nK}\right)$.
By definition:
$x_{i+1}^{k} = \Pi_{\mathcal{W}}\left( x_{i}^{k} - \alpha \nabla f(x_i^{k};\sigma_k(i+1))\right)$.

Taking norm squared and using Lemma~\ref{lem:convex_contraction}
\begin{align}
\|x_{i+1}^{k}-x^{*}\|^2 &\leq \|x_i^k-x^{*}\|^2 -2\alpha
\langle \nabla f(x^k_i;\sigma_k(i+1)),x_i^k-x^{*}\rangle +\alpha^2\|\nabla f(x^k_i;\sigma_k(i+1))\|^2 \nonumber \\
&\leq  \|x_i^k-x^{*}\|^2 -2\alpha
\langle \nabla f(x^k_i;\sigma_k(i+1)),x_i^k-x^{*}\rangle +\alpha^2G^2 \nonumber \\
&\leq  \|x_i^k-x^{*}\|^2 -2\alpha
\langle \nabla F(x^k_i),x_i^k-x^{*}\rangle +2\alpha
\langle \nabla F(x^k_i) - \nabla f(x^k_i;\sigma_k(i+1)),x_i^k-x^{*}\rangle +\alpha^2G^2 \nonumber \\
&\leq \|x_i^k-x^{*}\|^2(1-\alpha\mu) -2\alpha\left[F(x^k_i)- F(x^*)\right]+2\alpha
R_{i,k} +\alpha^2G^2 
\label{eq:fundamental_equation}
\end{align}
We have used strong convexity of $F(\cdot)$ in the fourth step. Here $R_{i,k} := \langle \nabla F(x^k_i) - \nabla f(x^k_i;\sigma_k(i+1)),x_i^k-x^{*}\rangle$. We will bound $\mathbb{E}[R_{i,k}]$.

Clearly, 
\begin{align*}
R_{i,k} &= \frac{1}{n}\sum_{r=1}^{n}\langle \nabla f(x^k_i;r) ,x_i^k-x^{*}\rangle -\langle  \nabla f(x^k_i;\sigma_k(i+1)),x_i^k-x^{*}\rangle
\end{align*}
Recall the definition of $\mathcal{D}_{i,k}$ and $\mathcal{D}_{i,k}^{(r)}$ from Section~\ref{sec:coupling_and_wasserstein}. Let $Y \sim \mathcal{D}_{i,k}$ and $Z_r \sim \mathcal{D}_{i,k}^{(r)}$, with any arbitrary coupling. Taking expecation in the expression for $R_{i,k}$, we have:

\begin{align*}
\mathbb{E}[R_{i,k}] &=  \frac{1}{n}\sum_{r=1}^{n} \mathbb{E}\left[\langle\nabla f(x^k_i;r) ,x_i^k-x^{*}\rangle\right] -\frac{1}{n}\sum_{r=1}^{n}\mathbb{E}\left[\langle  \nabla f(x^k_i;r),x_i^k-x^{*}\rangle\bigr| \sigma_k(i+1)= r\right] \\
&=  \frac{1}{n}\sum_{r=1}^{n}\mathbb{E} \left[\langle\nabla f(Y;r) ,Y-x^{*}\rangle-\langle  \nabla f(Z_r;r),Z_r-x^{*}\rangle\right] \\
&=  \frac{1}{n}\sum_{r=1}^{n}\mathbb{E} \biggr[\langle\nabla f(Y;r)-\nabla f(Z_r;r)  ,Y-x^{*}\rangle + \langle  \nabla f(Z_r;r),Y-Z_r\rangle\biggr] \\
&\leq \frac{1}{n}\sum_{r=1}^{n}\mathbb{E}[ L\|Y-x^*\|.\|Z_r-Y\| + G\|Z_r-Y\|] \\
&\leq \frac{1}{n}\sum_{r=1}^{n} L\sqrt{\mathbb{E}[\|Y-x^*\|^2]}\sqrt{\mathbb{E}[\|Z_r-Y\|^2]} + G\mathbb{E}[\|Z_r - Y\|]
\end{align*}
We have used smoothness of $f(;r)$ and Cauchy-Schwarz inequality in the fourth step and Cauchy-Schwarz inequality in the fifth step. Since the inequality above holds for every coupling between $Y$ and $Z_r$, we conclude:

\begin{align}
\mathbb{E}[R_{i,k}]  &\leq \frac{1}{n}\sum_{r=1}^{n} L \dwtwo\left(\mathcal{D}_{i,k},\mathcal{D}_{i,k}^{(r)}\right) \sqrt{\mathbb{E}[\|x_i^k-x^*\|^2]}+ G\dwtwo\left(\mathcal{D}_{i,k},\mathcal{D}_{i,k}^{(r)}\right)\nonumber \\
&\leq \frac{1}{n}\sum_{r=1}^{n} \frac{L^2}{\mu} \left[\dwtwo\left(\mathcal{D}_{i,k},\mathcal{D}_{i,k}^{(r)}\right)\right]^2  +  \frac{\mu}{4}\mathbb{E}[\|x_i^k-x^*\|^2] + G\dwtwo\left(\mathcal{D}_{i,k},\mathcal{D}_{i,k}^{(r)}\right)
\label{eq:some_intermediate_equation}
\end{align}

 by our hypethesis we have $\alpha \leq \frac{2}{L}$. So we can apply Lemma~\ref{lem:wasserstein_stability_bound}. Equation~\eqref{eq:some_intermediate_equation} along with equation~\eqref{eq:fundamental_equation} implies:

\begin{align*}
&\mathbb{E}\|x_{i+1}^{k}-x^{*}\|^2 \\&\leq   \mathbb{E}\|x_i^k-x^{*}\|^2(1-\alpha\mu) -2\alpha\mathbb{E}\left[F(x^k_i)- F(x^*)\right]+2\alpha\mathbb{E}
R_{i,k} +\alpha^2G^2  \\
&\leq  \mathbb{E}[\|x_i^k-x^{*}\|^2]\left(1-\tfrac{\alpha\mu}{2}\right)- 2\alpha\mathbb{E}\left[F(x^k_i)- F(x^*)\right] + 3G^2\alpha^2 + \frac{4L^2G^2\alpha^3}{\mu}
\end{align*}

We use the fact that $F(x^k_i)- F(x^*) \geq 0$ and unroll the recursion above to conclude:

\begin{align*}
\mathbb{E}[\|x_{0}^{k+1}-x^{*}\|^2] &\leq \left(1-\tfrac{\alpha\mu}{2}\right)^{nk}\|x_0^1-x^{*}\|^2 +\sum_{t=0}^{\infty}\left(1-\tfrac{\alpha\mu}{2}\right)^t\left[3G^2\alpha^2 + \tfrac{4L^2G^2\alpha^3}{\mu}\right]\\
&= \left(1-\tfrac{\alpha\mu}{2}\right)^{nk}\|x_0^1-x^{*}\|^2 +\left[\tfrac{6G^2\alpha}{\mu} + \tfrac{8L^2G^2\alpha^2}{\mu^2}\right]\\
&\leq e^{-\tfrac{n\alpha k \mu}{2}}\|x_0^1-x^{*}\|^2 +\left[\tfrac{6G^2\alpha}{\mu} + \tfrac{8L^2G^2\alpha^2}{\mu^2}\right]\\
\end{align*}

Using the fact that $\alpha =\min\left(\tfrac{2}{L},4l \tfrac{\log{nK}}{\mu nK}\right)$, we conclude that when $k \geq \frac{K}{2}$, 

\begin{equation}
\mathbb{E}[\|x_{0}^{k+1}-x^{*}\|^2] \leq \frac{\|x_0^1-x^{*}\|^2}{(nK)^l} + \left[\tfrac{6G^2\alpha}{\mu} + \tfrac{8L^2G^2\alpha^2}{\mu^2}\right]
\label{eq:distance_bound_midway}
\end{equation}

We can easily verify that equation~\ref{eq:no_strong_convex_equation} also holds in this case (because all other assumptions hold). Therefore, for $k \geq \frac{K}{2}$,

\begin{align*}
\mathbb{E}[\|x_{i+1}^{k}-x^{*}\|^2] &\leq \mathbb{E}[\|x_i^k-x^{*}\|^2] - 2\alpha \mathbb{E}[F(x^k_i)-F(x^{*})] + 5\alpha^2G^2
\end{align*}

Summing this equation for  $0\leq i \leq n-1$, $\lceil \tfrac{K}{2}\rceil \leq k\leq K$, we conclude:

\begin{align*}
 \tfrac{1}{n(K-\lceil\tfrac{K}{2}\rceil +1)}\sum_{k=\lceil\tfrac{K}{2}\rceil}^{K}\sum_{i=0}^{n-1} \mathbb{E}(F(x^k_i)-F(x^{*})) &\leq  \tfrac{1}{2n\alpha (K-\lceil\tfrac{K}{2}\rceil +1)}\mathbb{E}\bigr\|x_0^{\lceil\tfrac{K}{2}\rceil}-x^{*}\bigr\|^2  + \frac{5}{2}\alpha G^2
 \\&= O\left(\mu\tfrac{\|x_0^1-x^*\|^2}{(nK)^l} + L\tfrac{\|x_0^1-x^*\|^2}{(nK)^{(l+1)}}\right) + O\left(\tfrac{G^2\log{nK}}{\mu nK} + \tfrac{L^2G^2\log{nK}}{\mu^3 n^2K^2}\right)
\end{align*}
In the last step we have used Equation~\eqref{eq:distance_bound_midway} and the fact that $\alpha \leq \tfrac{4l \log{nK}}{\mu nK}$ and $\tfrac{1}{\alpha} \leq \tfrac{L}{2} + \tfrac{nK\mu }{4l\log{nK}}$. Using convexity of $F$, we conclude that:

$$F(\hat{x}) \leq \tfrac{1}{n(K-\lceil\tfrac{K}{2}\rceil +1)}\sum_{k=\lceil\tfrac{K}{2}\rceil}^{K}\sum_{i=0}^{n-1} F(x^k_i)\,.$$
This proves the result.
\section{Proofs of useful lemmas}\label{app:lem}
\begin{proof}[Proof of Lemma~\ref{lem:stability}]
	For simplicity of notation, we denote $y_i \defeq x_i(\sigma_k)$ and $z_i \defeq x_i(\sigma'_k)$.
	We know that $\|y_0-z_0\| =0$ almost surely by definition. Let $j < i$. First we Suppose $\tau_y(j+1) = r \neq s =  \tau_z(j+1)$. Then, by Lemma~\ref{lem:convex_contraction} 
	\begin{align*}
	&\|y_{j+1}-z_{j+1}\| \\&= \bigr\|\Pi_{\mathcal{W}}\left(y_j - \alpha_{k,j} \nabla f(y_j;r )\right) - \Pi_{\mathcal{W}}\left(z_j - \alpha_{k,j} \nabla f(z_j;s )\right)\bigr\| \\
	&\leq \|y_j -z_j - \alpha_{k,j} \left(\nabla f(y_j;r) -\nabla f(z_j;s ) \right)\|\\
	&\leq \|y_j - z_j\| + \alpha_{k,j}\|\nabla f(y_j;r)\| +\alpha_{k,j}\|\nabla f(z_j;s )\| \\
	&\leq 2G\alpha_{k,j} + \|y_j-z_j\| \\
	&\leq 2G\alpha_{k,0} + \|y_j-z_j\|
	\end{align*}
	In the last step above, we have used monotonicity of $\alpha_t$. Now, suppose $\tau_y(j+1) = \tau_z(j+1) = r$. Then,
	\begin{align*}
	\|y_{j+1}-z_{j+1}\|^2 &= \bigr\|\Pi_{\mathcal{W}}\left(y_j - \alpha_{k,j} \nabla f(y_j;r)\right) - \Pi_{\mathcal{W}}\left(z_j - \alpha_{k,j} \nabla f(z_j;r)\right)\bigr\|^2 \\
	&\leq \|\left(y_j - \alpha_{k,j} \nabla f(y_j;r)\right) - \left(z_j - \alpha_{k,j} \nabla f(z_j;r)\right)\|^2 \\
	&= \|y_j-z_j\|^2 -2\alpha_{k,i} \langle\nabla f(y_j;r) - \nabla f(z_j;r),y_j-z_j\rangle + \alpha_{k,j}^2\|\nabla f(y_j;r) - \nabla f(z_j;r)\|^2 \\
	&\leq \|y_j-z_j\|^2 -(2\alpha_{k,j} -  L\alpha_{k,j}^2) \langle\nabla f(y_j;r) - \nabla f(z_j;r),y_j-z_j\rangle \\
	&\leq \|y_j-z_j\|^2
	\end{align*}
	In the second step we have used Lemma~\ref{lem:smooth_convex_property} and in the third step we have used the fact that when $\alpha_{k,0} \leq \frac{2}{L}$, $2\alpha_{k,i} -  L\alpha_{k,i}^2 \geq 0$ and $\langle\nabla f(y_i;r) - \nabla f(z_i;r),y_i-z_i\rangle \geq 0$ by convexity. This proves the lemma.	
\end{proof}

\begin{proof}[Proof of Lemma~\ref{lem:temporal_regularity}]
	For the sake of clarity of notation, in this proof we take $R_{j} := \sigma_k(j)$ for all $j \in [n]$.
	By defintion, $x_{j+1}^k - x_0^k = \Pi_{\mathcal{W}} \left(x_j^k - \alpha_{k,j}\nabla f(x_j^k;R_{j+1})\right) - x_0^{k}$. Taking norm squared on both sides, we have:
	
	\begin{align*}
	&\|x_{j+1}^k - x_0^k\|^2 \\&\leq \|x_j^k - x_0^k\|^2 - 2 \alpha_{k,j}\langle f(x_j^k;R_{j+1}, x_j^k-x_0^k\rangle + \alpha^2_{k,j}G^2 \\
	&\leq  \|x_j^k - x_0^k\|^2 + 2 \alpha_{k,j}\left(f(x_0^k;R_{j+1}) - f(x_j^k;R_{j+1})\right)+ \alpha^2_{k,j}G^2
	\end{align*}
	Taking expectation on both sides, we have:
	\begin{align*}
	\mathbb{E}[\|x_{j+1}^k - x_0^k\|^2] &\leq \mathbb{E}[\|x_j^k - x_0^k\|^2] + \alpha^2_{k,j}G^2 +2 \alpha_{k,j}\mathbb{E}\left[f(x_0^k;R_{j+1}) - f(x_j^k;R_{j+1})\right]\\
	&= \mathbb{E}[\|x_j^k - x_0^k\|^2] +2 \alpha_{k,j}\mathbb{E}\left[F(x_0^k) - f(x_j^k;R_{j+1})\right] + \alpha^2_{k,j}G^2 \\
	&=  \mathbb{E}[\|x_j^k - x_0^k\|^2] +2 \alpha_{k,j}\mathbb{E}\left[F(x_0^k) - F(x_j^k)\right]+ 2\alpha_{k,j}\mathbb{E}\left[F(x_j^k)-f(x_j^k;R_{j+1})\right] + \alpha^2_{k,j}G^2 \\
	&\leq \mathbb{E}[\|x_j^k - x_0^k\|^2] +2 \alpha_{k,j}\mathbb{E}\left[F(x_0^k) - F(x_j^k)\right] + 4\alpha_{k,j}\alpha_{k,0}G^2 + \alpha^2_{k,j}G^2 \\
	&\leq \mathbb{E}[\|x_j^k - x_0^k\|^2] +2 \alpha_{k,j}\mathbb{E}\left[F(x_0^k) - F(x^*)\right] +4\alpha_{k,j}\alpha_{k,0}G^2 + \alpha^2_{k,j}G^2 \\
	\end{align*}
	In the fourth step we have used Lemma~\ref{lem:wasserstein_stability_bound} and in the fifth step, we have used the fact that $x^{*}$ is the minimizer of $F$. We sum the equation above from $j=0$ to $j=i-1$ and use the fact that $\alpha_{k,0}\geq \alpha_{k,j}$ and that $\|x_j^k -x_0^k\| =0$ when $j=0$ to conclude the result. For the proof of the second equation in the lemma, we use $x^*$ instead of $x_0^k$ above and go through similar steps.
\end{proof}

\end{document}